\documentclass{scrartcl}

\usepackage[expansion=false]{microtype}
\usepackage{url}
\usepackage{float}

\usepackage{amsmath}
\usepackage{amsthm}
\usepackage{amsfonts}
\usepackage{amssymb}

\usepackage{booktabs}
\usepackage{multirow}

\usepackage{enumitem}

\usepackage{geometry}
\usepackage[symbol,para]{footmisc}

\usepackage[numbers,sort&compress]{natbib}  

\usepackage[colorlinks=true,linkcolor=blue,citecolor=blue,urlcolor=blue]{hyperref}

\numberwithin{equation}{section}

\newtheorem{theorem}{Theorem}[section]
\newtheorem{lemma}[theorem]{Lemma}
\newtheorem{proposition}[theorem]{Proposition}
\newtheorem{corollary}[theorem]{Corollary}

\newtheorem{remark}[theorem]{Remark}

\theoremstyle{definition}

\makeatletter
\DeclareOldFontCommand{\rm}{\normalfont\rmfamily}{\mathrm}
\DeclareOldFontCommand{\sf}{\normalfont\sffamily}{\mathsf}
\DeclareOldFontCommand{\tt}{\normalfont\ttfamily}{\mathtt}
\DeclareOldFontCommand{\bf}{\normalfont\bfseries}{\mathbf}
\DeclareOldFontCommand{\it}{\normalfont\itshape}{\mathit}
\makeatother

\allowdisplaybreaks

\makeatletter
\g@addto@macro{\appendix}{%
}
\makeatother

\begin{document}

\title{On Circular Numerical Ranges of Companion Matrices with Repeated Eigenvalues}

\author{Hsin-Yi Lee\thanks{Key Laboratory of Applied Mathematics of Fujian Province University,
Putian University, Putian Fujian, 351100, China.
\href{mailto:mathlee1209@gmail.com}{mathlee1209@gmail.com}.}
\and
Wei-Qiang Huang\thanks{Center for Fundamental Science,
National Formosa University, Yunlin 632301, Taiwan.
\href{mailto:wqh@nfu.edu.tw}{wqh@nfu.edu.tw}.}}

\maketitle

\begin{abstract}
We prove that if an $n\times n\ (n > 3)$ companion matrix $A$ with the spectrum $\sigma(A) = \{ a \}$ 
has a circular numerical range, then $A$ is the Jordan block. This problem can be described by examining zeros of 
the Laurent polynomial arising from geometric properties of the numerical range. 
The difficulty is that the relevant Laurent coefficients involve both the repeated eigenvalue $a$ and the radius parameter $\lambda$, 
so direct coefficient comparison does not isolate $a$. We address this by decomposing the relevant matrix into a tridiagonal
Toeplitz part plus a rank-two update and using Chebyshev polynomials of the second kind. 
This reduction yields an explicit Laurent-coefficient formula whose vanishing under the circularity condition gives $a=0$. 
Furthermore, we extend this result when the spectrum is $\sigma(A) = \{0, a\}$ with algebraic multiplicities $n-m$ and $m$, respectively.
\end{abstract}

\newcommand{\keywords}[1]{\paragraph{Keywords:} #1}

\newcommand{\subclass}[1]{\paragraph{Mathematics Subject Classification (2020):} #1}
\subclass{15A60, 15A18, 05A10}

\section{Introduction}\label{sec:intro}

Let $A$ be an $n\times n$ complex matrix. Its numerical range is
\[
W(A)=\{\langle Ax,x\rangle:x\in\mathbb C^n,\ \|x\|=1\}.
\]
We write $\sigma(A)$ for the spectrum of $A$ and let $\operatorname{Re} A:= (A+A^{*})/2$. 
The set $W(A)$ is a compact convex subset of the complex plane, and its boundary is determined 
by the largest eigenvalues of the Hermitian matrices $\operatorname{Re}(e^{-i\theta}A)$ as $\theta$ varies over $[0,2\pi)$; 
see, for instance, \cite{HornJohnson91,GauWu21}. 
Given a monic polynomial
\[
p(z)=z^n+a_1z^{n-1}+\cdots+a_{n-1}z+a_n,
\]
we write $A$ for the companion matrix
\begin{equation}\label{eq:companion-form}
A=
\begin{bmatrix}
0 & 1 &        &        &   \\
  & 0 & 1      &        &   \\
  &   & \ddots & \ddots &   \\
  &   &        & 0      & 1 \\
-a_n & -a_{n-1} & \cdots & -a_2 & -a_1
\end{bmatrix}.
\end{equation}
It is a standard result that both the characteristic polynomial and the minimal polynomial 
of $A$ are equal to $p(z)$. When $p(z)=(z-a)^n$, the coefficients of $p(z)$ in 
\eqref{eq:companion-form} can be expressed as
\[
a_i=(-1)^i\binom{n}{i} a^i,\qquad 1\le i\le n.
\]
In particular, $A$ is the Jordan block $J_n$ if $a=0$.

Calbeck \cite{Calbeck08} showed that a $3\times 3$ companion matrix with 
$\sigma(A)=\{a\}$ has a circular numerical range if and only if either $a=0$ or $a = \pm \sqrt{2\sqrt{3}\cos 50^{\circ}-2}$. 
Based on this argument, he also asked whether a nonzero number $a$ exists in higher dimensions. 
We answer this question by proving that no such nonzero number $a$ exists for $n>3$. 
That is, if an $n\times n$ $(n>3)$ companion matrix $A$ with $\sigma(A)=\{a\}$ has a circular 
numerical range, then $a=0$ or, equivalently, $A=J_n$ (Theorem~\ref{thm:main}). 
Thus, within this class of companion matrices, the circular shape of $W(A)$ determines the algebraic parameter $a$ completely.

Let $A$ be an $n\times n$ companion matrix with 
$\sigma(A)=\{a\}$ having a circular numerical range $W(A)$. 
By Anderson's theorem, which is based on Kippenhahn's boundary-generating curve, 
the center of a circular numerical range of a finite matrix is an eigenvalue; 
see \cite{Kippenhahn51,Wu11,CheungLi13,GauWu21}. Since $\sigma(A)=\{a\}$, the center of $W(A)$ is $a$.
In addition, drawing on the work of Gau and Wu \cite[Lemma~2.8]{GauWu04}, we may assume $a\ge 0$ throughout this paper. 
We shall use the following elementary properties of the numerical range. For every 
square matrix $B$, every $b\in\mathbb C$, and every $|\omega|=1$,
\[
W(B-bI)=W(B)-b,\qquad W(\omega B)=\omega W(B),
\qquad W(\operatorname{Re} B)=\operatorname{Re} W(B).
\]
Therefore, if $W(A)$ is a circular disk centered at $a$ with radius $\lambda>0$, then
for every $|\omega|=1$,
\[
W\bigl(\operatorname{Re}(\omega(A-aI_n))\bigr)
=\operatorname{Re} W\bigl(\omega(A-aI_n)\bigr)
=[-\lambda,\lambda].
\]
Thus $\lambda$ is the largest eigenvalue of the Hermitian matrix 
$\operatorname{Re}(\omega(A-aI_n))$, and hence
\begin{equation}\label{eq:Aw}
A_\omega:=2\bigl(\lambda I_n-\operatorname{Re}(\omega(A-aI_n))\bigr)
\end{equation}
is singular for every $|\omega|=1$. The expansion of $\det(A_\omega)$ shows that 
$\omega^{n-2}\det(A_\omega)$ is a polynomial in $\omega$ of degree at most $2n-4$. 
Since this polynomial vanishes at every point on the unit circle, it is identically zero. 
Consequently, every Laurent coefficient of $\det(A_\omega)$ is zero. This coefficient 
vanishing is the starting point for the proof of $a=0$, equivalently $A=J_n$, for $n>3$.

The same coefficient-vanishing approach also handles companion matrices with partial zero 
spectra $\sigma(A)=\{0,a\}$ (Theorem~\ref{thm:partial}).

This paper is organized as follows. 
Section~\ref{sec:framework} derives the tridiagonal Toeplitz plus rank-two decomposition and the associated Chebyshev formulas. 
Section~\ref{sec:structural} establishes the degree bounds and the vanishing of leading coefficients. 
Section~\ref{sec:proof} identifies the decisive Laurent coefficient and proves the main theorem. 
Section~\ref{sec:partial-zeros} treats the partial zero-spectrum case. 
The required binomial identities are collected in Appendix~\ref{app:identities}.

\section{Tridiagonal Toeplitz Matrix with Rank-Two Update}\label{sec:framework}

Since $A_\omega$ defined in \eqref{eq:Aw} is singular for every $|\omega|=1$, we have
\begin{equation}\label{eq:det_zero}
\det(A_\omega) = 2^n\cdot\det\bigl(\lambda I_n-\operatorname{Re}(\omega(A-aI_n))\bigr) = 0.
\end{equation}
The matrix $A_\omega$ can be expressed as a tridiagonal Toeplitz matrix with a rank-two update:
\begin{align}
A_\omega
&=\begin{bmatrix}
2\mu            & -\omega       &        &                                & a_n\overline{\omega}\\
-\overline{\omega}& 2\mu            & -\omega&                                & a_{n-1}\overline{\omega}\\
              & -\overline{\omega}& \ddots & \ddots                         & \vdots\\
              &               & \ddots & 2\mu                             & a_2\overline{\omega}-\omega\\
a_n\omega     & a_{n-1}\omega & \cdots & a_2\omega-\overline{\omega}    & 2\mu+a_1(\omega+\overline{\omega})
\end{bmatrix}\notag\\
&=T_{\mu}+
\begin{bmatrix}\alpha & e_n\end{bmatrix}
\begin{bmatrix}0 & \overline{\omega}\\ \omega & 0\end{bmatrix}
\begin{bmatrix}\alpha^{\top}\\ e_n^{\top}\end{bmatrix}.\label{eq:Td-rank2}
\end{align}
Here $T_{\mu}$ is a tridiagonal Toeplitz matrix:
\begin{equation}\label{eq:Td}
T_{\mu}=
\begin{bmatrix}
2\mu            & -\omega           &        &                    &        \\
-\overline{\omega}& 2\mu            & -\omega&                    &        \\
              & -\overline{\omega}& \ddots & \ddots             &        \\
              &                   & \ddots & 2\mu                 & -\omega\\
              &                   &        & -\overline{\omega} & 2\mu
\end{bmatrix},
\end{equation}
where $\mu = \lambda + a(\omega+\overline{\omega})/2$, $\alpha=[a_n,\ldots,a_2,a_1]^{\top}$ and $e_{n}=[0,\ldots,0,1]^{\top}$. Note that $\alpha_i = a_{n+1-i}$ for $1\le i\le n$.

To analyze \eqref{eq:det_zero} via the rank-two structure of \eqref{eq:Td-rank2}, 
we apply the rank-update determinant formula \cite{Harville97}: for any nonsingular 
$M\in\mathbb{C}^{n\times n}$, $U,V\in\mathbb{C}^{n\times 2}$ and $C\in\mathbb{C}^{2\times 2}$,
\begin{equation}\label{eq:detlemma}
   \det\bigl(M + U\,C\,V^{\top}\bigr) = \det(M) \cdot \det\bigl(I_{2} + C\,V^{\top} M^{-1} U\bigr),
\end{equation}
which requires $T_{\mu}$ (playing the role of $M$) to be nonsingular. 
We will see that $\det(T_{\mu})$ can be expressed in terms of a Chebyshev polynomial of the second kind. 
On the unit circle, we write $\omega=e^{-i\theta}$, so that $\overline{\omega}=e^{i\theta}$ and $\omega+\overline{\omega}=2\cos\theta$. 
Hence the parameter
\[
   \mu = \lambda + \frac{a}{2}(\omega + \overline{\omega}) = \lambda + a\cos\theta
\]
is real as well. 
We introduce the Chebyshev polynomials of the second kind \cite{MasonHandscomb2002}, defined by
\begin{equation}\label{eq:cheby_def}
U_n(x)=\sum_{i=0}^{\lfloor n/2\rfloor}(-1)^i\binom{n-i}{i}(2x)^{n-2i}.
\end{equation}
The classical zeros of $U_n$ are $\mu = \cos(k\pi/(n+1))$ for $k = 1, \ldots, n$ \cite{MasonHandscomb2002}, 
all lying in $(-1, 1)$. Hence $T_{\mu}$ is singular precisely at those $\omega = e^{-i\theta}$ with $|\omega|=1$ satisfying
\begin{equation*}
  \lambda + a\cos\theta = \cos\frac{k\pi}{n+1}
  \quad\text{for some } k \in \{1,\ldots,n\}.
\end{equation*}
For $a \neq 0$, this equation has at most $2n$ solutions for $\theta \in [0,2\pi)$,
so $\{\omega : |\omega|=1,\ \det T_{\mu} \ne 0\}$ is the complement of a finite set.
The proofs in Sections~\ref{sec:proof}--\ref{sec:partial-zeros} proceed under the 
assumption $a > 0$ (by the rotation reduction of Section~\ref{sec:intro}), so this finiteness condition holds.
 
We first exclude the finitely many $\omega$ on the unit circle at which $T_\mu$ is singular, 
so that $T_\mu^{-1}$ exists 
(the formula is extended to all $|\omega|=1$ in Remark~\ref{rem:formula-extends} below). 
Combining \eqref{eq:detlemma} with the rank-two decomposition \eqref{eq:Td-rank2} 
then gives the following closed form for $\det(A_\omega)$:

\begin{align}
\det(A_\omega)
=&\ \det\left( T_{\mu} +  \begin{bmatrix} \alpha & e_{n}\end{bmatrix}
                        \begin{bmatrix} 0 & \overline{\omega } \\ \omega & 0\end{bmatrix}
                        \begin{bmatrix} \alpha^\top \\ e_{n}^\top\end{bmatrix}\right)  \notag \\
=&\ \det(T_{\mu})\cdot \det\left(I_{2} + 
                        \begin{bmatrix} 0 & \overline{\omega } \\ \omega & 0\end{bmatrix}\begin{bmatrix} \alpha^\top \\ e_{n}^\top\end{bmatrix} T_{\mu}^{-1}\begin{bmatrix} \alpha & e_{n}\end{bmatrix}\right)  \notag \\
=&\ \det(T_{\mu})\cdot (-1)\cdot \det\left(\begin{bmatrix} 0 & \overline{\omega } \\ \omega & 0\end{bmatrix} + 
                        \begin{bmatrix} \alpha^\top \\ e_{n}^\top\end{bmatrix} T_{\mu}^{-1}\begin{bmatrix} \alpha & e_{n}\end{bmatrix}\right)  \notag \\
=&\ (-\det(T_{\mu})) \cdot \det\left(\begin{bmatrix} \alpha ^{\top}T_{\mu}^{-1}\alpha & \overline{\omega}+\alpha ^{\top}T_{\mu}^{-1}e_{n}\\ 
                                             \omega +e_{n}^{\top}T_{\mu}^{-1}\alpha & e_{n}^{\top}T_{\mu}^{-1}e_{n}
                           \end{bmatrix}\right) \notag\\
=& \ \det(T_{\mu}) + \det(T_{\mu})(\alpha^{\top}T_{\mu}^{-1}e_{n}) \omega + 
     \det(T_{\mu})(e_{n}^{\top}T_{\mu}^{-1}\alpha)\overline{\omega} \notag\\
& \qquad\qquad - \det(T_{\mu})\Big((\alpha^{\top}T_{\mu}^{-1}\alpha)(e_{n}^{\top}T_{\mu}^{-1}e_{n}) - (\alpha^{\top}T_{\mu}^{-1}e_{n})(e_{n}^{\top}T_{\mu}^{-1}\alpha)\Big).  \label{eq:detA_four_terms}
\end{align}

To simplify notation, we denote
\begin{equation*}
u_n:=\det(T_{\mu}) \quad\text{and}\quad H_{\mu}:=u_n\,T_{\mu}^{-1},
\end{equation*}
representing the determinant of $T_{\mu}$ and its inverse scaled by that determinant. 
In the next two subsections, we will derive closed-form expressions for $u_n$ and $H_{\mu}$. 

\subsection[The Determinant of Tmu]{The Determinant of \texorpdfstring{$T_{\mu}$}{T_mu}}

We now show that $\det(T_{\mu})$ equals $U_n(\mu)$.

\begin{lemma}\label{lem:detTd}
The determinant of the $n\times n$ tridiagonal matrix $T_{\mu}$ equals the $n$th Chebyshev polynomial of the second kind:
\[
\det(T_{\mu})=u_n=U_n(\mu).
\]
\end{lemma}
\begin{proof}
Let $u_k$ denote the determinant of the $k\times k$ leading principal submatrix of $T_{\mu}$ in \eqref{eq:Td}, with initial values $u_0=1$ and $u_1=2\mu$. Cofactor expansion along the last row yields the three-term recurrence
\[
u_k=2\mu\cdot u_{k-1}-u_{k-2},\qquad k\ge 2,
\]
which in matrix form gives
\[
\begin{bmatrix}u_n\\u_{n-1}\end{bmatrix}
=\begin{bmatrix}2\mu & -1\\ 1 & 0\end{bmatrix}^n
\begin{bmatrix}1\\0\end{bmatrix}.
\]
By \cite{ChenLouck96}, the $n$th power of the $2\times 2$ companion matrix 
$\left(\begin{smallmatrix} 2\mu & -1 \\ 1 & 0 \end{smallmatrix}\right)$ has entries that are Chebyshev polynomials of the second kind:
\[
\begin{bmatrix}2\mu & -1\\ 1 & 0\end{bmatrix}^n=\begin{bmatrix}U_n(\mu) & -U_{n-1}(\mu)\\ U_{n-1}(\mu) & -U_{n-2}(\mu)\end{bmatrix}.
\]
In particular, $\det(T_{\mu}) = u_n = U_n(\mu)$. \qedhere
\end{proof}

Substituting $\mu=\lambda+a(\omega+\overline{\omega})/2$ into $U_{n}(\mu)$ yields the explicit Laurent expansion.

\begin{corollary}\label{cor:un}
The determinant $\det(T_{\mu}) = u_{n}$ has the expansion
\begin{equation}\label{eq:un-expansion}
u_n = \sum_{i=0}^{\lfloor n/2\rfloor}\sum_{j=0}^{n-2i}\sum_{k=0}^{n-2i-j}(-1)^i\binom{n-i}{i,j,k,n-2i-j-k}(2\lambda)^j a^{n-2i-j}\omega^{n-2i-j-2k},
\end{equation}
where $$\binom{n-i}{i,j,k,n-2i-j-k} = \frac{(n-i)!}{i!\ j!\ k!\ (n-2i-j-k)!}$$ is the multinomial coefficient 
with $0\le i\le\lfloor n/2\rfloor$, $0\le j\le n-2i$, and $0\le k\le n-2i-j$.
\end{corollary}

\begin{proof}
According to Lemma~\ref{lem:detTd}, the definition \eqref{eq:cheby_def} and 
the value $\mu=\lambda + a(\omega +\overline{\omega })/2$, we have
\begin{align*}
\det (T_{\mu}) &=U_{n}(\mu) \\
&=\sum_{i=0}^{\lfloor n/2\rfloor}(-1)^{i}\tbinom{n-i}{i}( 2\mu)^{n-2i} \\
&=\sum_{i=0}^{\lfloor n/2\rfloor}(-1)^{i}\tbinom{n-i}{i}( 2\lambda + a(\omega +\overline{\omega })) ^{n-2i} \\
&=\sum_{i=0}^{\lfloor n/2\rfloor}(-1)^{i}\tbinom{n-i}{i}\left(\sum_{j=0}^{n-2i}\tbinom{n-2i}{j}(2\lambda)^{j}a^{n-2i-j}(\omega +\overline{\omega })^{n-2i-j}\right) \\
&=\sum_{i=0}^{\lfloor n/2\rfloor}(-1)^{i}\tbinom{n-i}{i}\left(\sum_{j=0}^{n-2i}\tbinom{n-2i}{j}(2\lambda)^{j}a^{n-2i-j}\left( \sum_{k=0}^{n-2i-j}\tbinom{n-2i-j}{k}\omega^{n-2i-j-k}\overline{\omega }^{k}\right)\right) \\
&=\sum_{i=0}^{\lfloor n/2\rfloor}\sum_{j=0}^{n-2i}\sum_{k=0}^{n-2i-j}(-1)^{i}\tbinom{n-i}{i}\tbinom{n-2i}{j}\tbinom{n-2i-j}{k}(2\lambda)^{j}a^{n-2i-j}\omega ^{n-2i-j-2k} \\
&=\sum_{i=0}^{\lfloor n/2\rfloor}\sum_{j=0}^{n-2i}\sum_{k=0}^{n-2i-j}(-1)^{i}\frac{(n-i)!}{i!j!k!(n-2i-j-k)!}(2\lambda)^{j}a^{n-2i-j}\omega ^{n-2i-j-2k} \\
&=\sum_{i=0}^{\lfloor n/2\rfloor}\sum_{j=0}^{n-2i}\sum_{k=0}^{n-2i-j}(-1)^{i}\binom{n-i}{i,j,k,n-2i-j-k}(2\lambda)^{j}a^{n-2i-j}\omega^{n-2i-j-2k}
\end{align*}
where $$\binom{n-i}{i,j,k,n-2i-j-k} = \frac{(n-i)!}{i!\ j!\ k!\ (n-2i-j-k)!}$$ is the multinomial coefficient and 
$i,j,k$ are nonnegative integers satisfying $0\leq i\leq \lfloor n/2\rfloor $, $0\leq j\leq
n-2i $ and $0\leq k\leq n-2i-j$. \qedhere
\end{proof}

\begin{corollary}\label{cor:un_leading_coeff}
The coefficients of $\omega^s$ and $\omega^{s-1}$ in $u_s$ are respectively given by $a^s$ and $2s\lambda a^{s-1}$ for all $1 \le s \le n$.
\end{corollary}

\begin{proof}
According to Corollary~\ref{cor:un}, the contribution of $u_s = U_s(\lambda + a(\omega + \overline{\omega})/2)$ 
to the coefficient of $\omega^p$ comes from triples $(i,j,k)$ with $2i+j+2k = s-p$. 
For $1\le s\le n$, setting $p=s$ and $p=s-1$ leads to $(i,j,k) = (0,0,0)$ and $(0,1,0)$ respectively. 
Hence, in $u_s$, the coefficients of $\omega^s$ and $\omega^{s-1}$ are $a^s$ and $2s\lambda a^{s-1}$, respectively.\qedhere
\end{proof}

\subsection{The Inverse of \texorpdfstring{$T_{\mu}$}{T_mu}}

The inverse of $T_{\mu}$ is obtained from the tridiagonal inverse formula proposed by da Fonseca and Petronilho \cite{daFonPetro01}.

\begin{lemma}\label{lem:invTd}
The inverse of the tridiagonal matrix $T_{\mu}$ is
\[
T_{\mu}^{-1}=\frac{1}{u_n}H_{\mu},
\]
where $H_{\mu}$ is the centro-Hermitian matrix\footnote{$H_{\mu}$ satisfies $\mathcal{J}H_{\mu}\mathcal{J}=\overline{H}_{\mu}$, 
where $\mathcal{J}$ is the $n\times n$ exchange matrix.}
\begin{equation}\label{eq:Hd}
H_\mu=\begin{bmatrix}
u_{n-1}                    & u_{n-2}\omega                    & u_{n-3}\omega^2                  & \cdots & u_2\omega^{n-3}                  & u_1\omega^{n-2}                  & \omega^{n-1}\\
u_{n-2}\overline{\omega}   & u_{1}u_{n-2}                     & u_{1}u_{n-3}\omega               & \cdots & u_{1}u_{2}\omega^{n-4}           & u_{1}u_{1}\omega^{n-3}           & u_1\omega^{n-2}\\
u_{n-3}\overline{\omega}^2 & u_{1}u_{n-3}\overline{\omega}                     & u_{2}u_{n-3}                     & \cdots & u_{2}u_{2}\omega^{n-5}           & u_{2}u_{1}\omega^{n-4}           & u_2\omega^{n-3}\\
\vdots                     & \vdots                           & \vdots                           & \ddots & \vdots                           & \vdots                           & \vdots\\
u_2\overline{\omega}^{n-3} & u_{1}u_{2}\overline{\omega}^{n-4} & u_{2}u_{2}\overline{\omega}^{n-5} & \cdots & u_{n-3}u_2                       & u_{n-3}u_1\omega                 & u_{n-3}\omega^2\\
u_1\overline{\omega}^{n-2} & u_{1}u_{1}\overline{\omega}^{n-3} & u_{2}u_{1}\overline{\omega}^{n-4} & \cdots & u_{n-3}u_1\overline{\omega}      & u_{n-2}u_1                       & u_{n-2}\omega\\
\overline{\omega}^{n-1}    & u_1\overline{\omega}^{n-2}        & u_2\overline{\omega}^{n-3}        & \cdots & u_{n-3}\overline{\omega}^2        & u_{n-2}\overline{\omega}         & u_{n-1}
\end{bmatrix}.
\end{equation}
Here, the elements are expressed in terms of the second-kind Chebyshev polynomials, where $u_k = U_k(\mu)$ for $k \ge 1$ and $u_0 = 1$.
\end{lemma}

\begin{proof}
By \cite[Corollary~4.1]{daFonPetro01}, the entries of the inverse of the tridiagonal Toeplitz matrix $T_{\mu}$ in \eqref{eq:Td} are given by
\begin{equation*}
\left( T_{\mu}^{-1}\right) _{ij}=\left\{ 
\begin{array}{cc}
    \dfrac{u_{i-1}u_{n-j}}{u_{n}}\omega ^{j-i} & \text{if }i\leq j \medskip\\ 
    \dfrac{u_{j-1}u_{n-i}}{u_{n}}\overline{\omega }^{i-j} & \text{if }i>j%
\end{array}%
\right.
\end{equation*}
where $u_k$ is the $k$th Chebyshev polynomial of the second kind. \qedhere
\end{proof}

Combining the identity $u_n(e_{n}^{\top}T_{\mu}^{-1}e_{n}) = H_{\mu}(n,n) = u_{n-1}$ from Lemma~\ref{lem:invTd} with \eqref{eq:detA_four_terms}, 
$\det(A_{\omega})$ admits the expression
 \begin{equation}\label{eq:WA_result}
  \det(A_{\omega}) = u_n+(\alpha^{\top} H_{\mu} e_n)\omega+(e_n^{\top} H_{\mu}\alpha)\overline{\omega}
   + \underbrace{\frac{1}{u_n}\Big((\alpha^{\top} H_{\mu} e_n)(e_n^{\top} H_{\mu}\alpha)-(\alpha^{\top} H_{\mu}\alpha)u_{n-1}\Big)}_{\displaystyle\mathcal{D}_4}.
\end{equation}

\section{Degree Reduction and Leading Coefficients Vanishing}\label{sec:structural}

This section rewrites the four-term decomposition \eqref{eq:WA_result} in Laurent-polynomial form.
Only $\mathcal D_4$ needs further reduction: it carries the denominator $u_n$ but a Chebyshev product-to-sum identity
cancels this denominator and shows that its positive $\omega$-degree is at most $n-2$.

\subsection{Structural Properties of \texorpdfstring{$\mathcal{D}_4$}{D4}}\label{subsec:D4-structure}

\begin{lemma}\label{lem:D4-formula}
The last term $\mathcal{D}_4$ in \eqref{eq:WA_result} can be explicitly expressed as the following Laurent polynomial in $\omega$:
\begin{equation}\label{eq:D4-combine}
\mathcal{D}_4 =
-\sum_{s=2}^{n}\sum_{t=1}^{s-1}
  \alpha_{s-t}\,\alpha_{n-t}\,u_{s-t-1}\,u_{t-1}\;\omega^{n-s}
-\sum_{s=2}^{n-1}\sum_{t=1}^{s-1}
  \alpha_{s-t}\,\alpha_{n-t}\,u_{s-t-1}\,u_{t-1}\;\overline{\omega}^{\,n-s}.
\end{equation}
\end{lemma}

\begin{proof}
Expanding the numerator of $\mathcal{D}_4$ in the components of $\alpha$, we have
\begin{equation}\label{eq:D4-numerator-bilinear}
(\alpha^{\top} H_{\mu} e_n)(e_n^{\top} H_{\mu} \alpha) - (\alpha^{\top} H_{\mu} \alpha)\,u_{n-1}
= \sum_{i,j=1}^n \alpha_i\alpha_j \Bigl( H_{\mu}(i,n)H_{\mu}(n,j)-H_{\mu}(i,j)u_{n-1} \Bigr).
\end{equation}
By equation \eqref{eq:Hd}, we first consider the case $1\le i\le j\le n$.
The parentheses in equation \eqref{eq:D4-numerator-bilinear} can be written as follows:
\begin{align}\label{eq:D4-ij-before-cancel}
H_{\mu}(i,n)H_{\mu}(n,j)-H_{\mu}(i,j)u_{n-1}
&= u_{i-1}\omega^{n-i}\cdot u_{j-1}\overline{\omega}^{n-j} - u_{i-1}u_{n-j}\omega^{j-i}\cdot u_{n-1} \notag\\
&= u_{i-1} \bigl( u_{j-1}-u_{n-j}u_{n-1} \bigr) \omega^{j-i}.
\end{align}
The Chebyshev product-to-sum identity
\begin{equation}\label{eq:cheb-pro2sum}
  u_{s}\,u_{t} = \sum_{r=0}^{\min\{s,t\}} u_{s+t-2r}, \quad s,t\ge 0,
\end{equation}
which follows from $U_k(\cos\theta) = \sin((k{+}1)\theta)/\sin\theta$ and the product-to-sum formula
$\sin\phi\sin\psi=(\cos(\phi-\psi)-\cos(\phi+\psi))/2$
(cf.\ \cite{MasonHandscomb2002}), together with the convention $u_{-1}=0$, gives
\[
u_{n-j}u_{n-1}
=\sum_{r=0}^{n-j} u_{2n-j-1-2r}
= u_{j-1} + \underbrace{\sum_{r=0}^{n-j-1} u_{2n-j-1-2r}}_{\displaystyle = u_{n}u_{n-j-1}}.
\]
Here the last term, corresponding to $r=n-j$, is $u_{j-1}$,
and the underbraced equality follows by applying the same product-to-sum identity
\eqref{eq:cheb-pro2sum} to $u_nu_{n-j-1}$.
When $j=n$, the underbraced sum is interpreted as an empty sum and agrees with $u_nu_{-1}=0$ under the convention $u_{-1}=0$.
Hence
\begin{align*}
u_{i-1} \bigl( u_{j-1}-u_{n-j}u_{n-1} \bigr)
= u_{i-1} \left(-\sum_{r=0}^{n-j-1} u_{2n-j-1-2r}\right)
= u_{i-1} \big(-u_{n}u_{n-j-1}\big)
= -u_{n}u_{i-1}u_{n-j-1}.
\end{align*}
Consequently, \eqref{eq:D4-ij-before-cancel} reduces to
\begin{equation*}
H_{\mu}(i,n)H_{\mu}(n,j)-H_{\mu}(i,j)u_{n-1} = -u_nu_{i-1}u_{n-j-1}\omega^{j-i}.
\end{equation*}
The same computation, with the roles of $i$ and $j$ interchanged, gives that for $1\le j<i\le n$,
\begin{equation*}
H_{\mu}(i,n)H_{\mu}(n,j)-H_{\mu}(i,j)u_{n-1} = -u_nu_{j-1}u_{n-i-1}\overline{\omega}^{\,i-j}.
\end{equation*}
Combining these two cases in \eqref{eq:D4-numerator-bilinear}, and then dividing by $u_n$, we get
\begin{equation}\label{eq:D4-piecewise-alpha}
\mathcal{D}_4 = -\sum_{1\le i\le j\le n}\alpha_i\alpha_j u_{i-1}u_{n-j-1}\omega^{j-i} - \sum_{1\le j<i\le n} \alpha_i\alpha_j u_{j-1}u_{n-i-1}\overline{\omega}^{\,i-j}.
\end{equation}

We now reindex the first sum in \eqref{eq:D4-piecewise-alpha} so as to
organize its Laurent expansion according to the powers $\omega^{n-s}$.
Set $i=s-t$ and $j=n-t$, where $2\le s\le n$ and $1\le t\le s-1$.
Then $j-i=n-s$ and $u_{i-1}u_{n-j-1}=u_{s-t-1}u_{t-1}$.
Therefore the first sum in \eqref{eq:D4-piecewise-alpha} becomes
\begin{equation}\label{eq:D4-positive-double-direct}
-\sum_{s=2}^{n}\sum_{t=1}^{s-1}
  \alpha_{s-t}\,\alpha_{n-t}\,u_{s-t-1}\,u_{t-1}\;\omega^{n-s}.
\end{equation}
The same reindexing applied to the second sum in \eqref{eq:D4-piecewise-alpha} gives
\begin{equation}\label{eq:D4-negative-double-direct}
-\sum_{s=2}^{n-1}\sum_{t=1}^{s-1}
  \alpha_{s-t}\,\alpha_{n-t}\,u_{s-t-1}\,u_{t-1}\;\overline{\omega}^{\,n-s}.
\end{equation}
Combining \eqref{eq:D4-positive-double-direct} and
\eqref{eq:D4-negative-double-direct} yields the explicit expression \eqref{eq:D4-combine}.
\end{proof}

\begin{corollary}\label{cor:D4-degree}
The term $\mathcal{D}_4$ is a Laurent polynomial in $\omega$ with positive $\omega$-degree at most $n-2$.
\end{corollary}

\begin{proof}
Recall that the parameter $\mu = \lambda + a(\omega+\overline{\omega})/2$ ensures that the highest positive $\omega$-power in the Chebyshev polynomial $u_k = U_{k}(\mu)$ is at most $k$. We analyze the maximum positive $\omega$-degree for each sum in the explicit formula \eqref{eq:D4-combine}.

For the positive $\omega$-powers in the first sum, the maximum degree of each summand is the sum of the degrees of its components ($u_{s-t-1}$, $u_{t-1}$, and $\omega^{n-s}$):
\begin{equation*}
(s-t-1) + (t-1) + (n-s) = n-2.
\end{equation*}

For the negative $\omega$-powers involving $\overline{\omega}^{\,n-s} = \omega^{s-n}$ in the second sum, the maximum positive $\omega$-degree of each term is similarly bounded above by $n-4$:
\begin{equation*}
(s-t-1) + (t-1) + (s-n) = 2s - n - 2 \le 2(n-1) - n - 2 = n - 4.
\end{equation*} 

Thus, all terms in $\mathcal{D}_4$ have a positive $\omega$-degree bounded by $n-2$.
\end{proof}

\begin{remark}\label{rem:formula-extends}
The closed form \eqref{eq:WA_result} was derived under the assumption that $T_\mu$ is nonsingular.
By Corollary~\ref{cor:D4-degree},
however, the right-hand side is a Laurent polynomial in $\omega$, as is $\det(A_\omega)$ itself.
Since two Laurent polynomials that agree on a cofinite subset of the unit circle must agree identically,
\eqref{eq:WA_result} holds for every $|\omega|=1$, including the finitely many excluded values.
\end{remark}

We now compute the coefficients of $\mathcal{D}_4$ at $\omega^{n-2}$ and $\omega^{n-3}$
under the single-eigenvalue specialization $a_r=(-1)^r\binom{n}{r}a^r$,
$1\le r\le n$, introducing two integer arrays used throughout the rest of the paper.

\begin{lemma}\label{lem:D4-coef}
Assume $a_r=(-1)^r\binom{n}{r}a^r$ for $1\le r\le n$, and let $\mathcal{D}_4$ be the explicit formula \eqref{eq:D4-combine} derived in Lemma~\ref{lem:D4-formula}.
For $n\ge 2$ and $1\le k\le n$ (with the convention $\binom{n}{-1}=0$), define
\begin{equation}\label{eq:cd-def}
T(n,k):=\binom{n-1}{k-1}\binom{n}{k},
\qquad
S(n,k):=\Biggl((n-1)\binom{n-1}{k-1}-\binom{n-2}{k-2}\Biggr)\binom{n}{k}.
\end{equation}
Then the coefficient of $\omega^{n-2}$ in $\mathcal{D}_4$ is
\begin{equation}\label{eq:D4-000}
-a^{n-2}\sum_{k=2}^{n}T(n,k)\,a^{2k},
\end{equation}
and the coefficient of $\omega^{n-3}$ in $\mathcal{D}_4$ is
\begin{equation}\label{eq:D4-010}
-2\lambda\,a^{n-3}\sum_{k=2}^{n}S(n,k)\,a^{2k}.
\end{equation}
\end{lemma}

\begin{proof}
Applying \eqref{eq:cheb-pro2sum} to $u_{s-t-1}\,u_{t-1}$
in \eqref{eq:D4-positive-double-direct}, the positive-$\omega$ part of $\mathcal{D}_4$ expands to
\begin{equation}\label{eq:D4-positive-triple}
-\sum_{s=2}^{n}\sum_{t=1}^{s-1}\sum_{r=0}^{\min(s-t-1,\,t-1)}
  \alpha_{s-t}\,\alpha_{n-t}\,u_{s-2-2r}\;\omega^{n-s}.
\end{equation}
The $\overline{\omega}$-part of $\mathcal{D}_4$ contributes only to positive
degrees at most $n-4$ (Corollary~\ref{cor:D4-degree}), hence is irrelevant for
both target exponents $\omega^{n-2}$ and $\omega^{n-3}$.

\textit{Formula \eqref{eq:D4-000}.}
The target $\omega^{n-2}$ requires $r=0$ in \eqref{eq:D4-positive-triple}.
By \eqref{eq:un-expansion} with triple $(0,0,0)$,
the $\omega^{s-2}$-coefficient of $u_{s-2}$ is $a^{s-2}$,
so the contribution of $\mathcal{D}_4$ to the coefficient of $\omega^{n-2}$ is
\[
-\sum_{s=2}^{n}a^{s-2}\left(\sum_{t=1}^{s-1}\alpha_{s-t}\,\alpha_{n-t}\right).
\]
Substituting $\alpha_i=(-1)^{n+1-i}\binom{n}{n+1-i}a^{n+1-i}$,
setting $q=s-t-1$, and reindexing $k=s$ gives
\[
-a^{n-2}\sum_{k=2}^{n}(-1)^{n+k}\binom{n}{k}a^{2k}
   \Biggl(\sum_{q=0}^{n-k}(-1)^q\binom{n}{q}\Biggr).
\]
By \eqref{eq:partial-altsum} the inner sum equals $(-1)^{n-k}\binom{n-1}{k-1}$,
giving
\[
-a^{n-2}\sum_{k=2}^{n}\binom{n-1}{k-1}\binom{n}{k}\,a^{2k}
= -a^{n-2}\sum_{k=2}^{n}T(n,k)\,a^{2k},
\]
which is \eqref{eq:D4-000}.

\textit{Formula \eqref{eq:D4-010}.}
The target $\omega^{n-3}$ requires $r=0$ and the $(0,1,0)$-coefficient of $u_{s-2}$,
which is $2(s-2)\lambda a^{s-3}$ by \eqref{eq:un-expansion}; this vanishes at $s=2$,
so the sum starts at $s=3$:
\[
-2\lambda\sum_{s=3}^{n}(s-2)a^{s-3}\left(\sum_{t=1}^{s-1}\alpha_{s-t}\,\alpha_{n-t}\right).
\]
The same substitution and reindexing as above gives
\[
-2\lambda a^{n-3}\sum_{k=2}^{n}(-1)^{n+k}\binom{n}{k}a^{2k}
\Biggl(\sum_{q=0}^{n-k}q(-1)^q\binom{n}{q}
+(k-2)\sum_{q=0}^{n-k}(-1)^q\binom{n}{q}\Biggr).
\]
Applying \eqref{eq:partial-altsum-firstmoment} and \eqref{eq:partial-altsum}
with upper bound $n-k$, and using
$\binom{n-2}{n-k-1}=\binom{n-2}{k-1}$, $\binom{n-1}{n-k}=\binom{n-1}{k-1}$,
the expression in parentheses evaluates to
$(-1)^{n-k}\bigl(n\binom{n-2}{k-1}+(k-2)\binom{n-1}{k-1}\bigr)$.
Applying \eqref{eq:D4-binomial-reduction} gives
\[
-2\lambda a^{n-3}\sum_{k=2}^{n}
\left((n-1)\binom{n-1}{k-1}-\binom{n-2}{k-2}\right)\binom{n}{k}\,a^{2k}
= -2\lambda a^{n-3}\sum_{k=2}^{n}S(n,k)\,a^{2k},
\]
which is \eqref{eq:D4-010}. \qedhere
\end{proof}

\begin{remark}\label{rem:OEIS}
The array $T(n,k)=\binom{n-1}{k-1}\binom{n}{k}$ appears, up to a shift of
indices, as OEIS\footnote{OEIS denotes \emph{The On-Line Encyclopedia of
Integer Sequences}; see \url{https://oeis.org}.} entry A103371
\cite{OEIS103371}. Moreover, $S(n,k)$ is a weighted variant of $T(n,k)$:
\[
S(n,k)=\left(n-1 - \frac{k-1}{n-1}\right) T(n,k),
\]
where the integrality of $S(n,k)$ follows from its definition \eqref{eq:cd-def},
since $(k-1)/(n-1)\binom{n-1}{k-1}=\binom{n-2}{k-2}$ is an integer.
In particular, $S(n,k)$ is a strictly positive integer for $n\ge 3$ and $2\le k\le n$;
note that both $T(n,k)$ and $S(n,k)$ are defined purely combinatorially by
\eqref{eq:cd-def} and depend on $n$ and $k$ only, with no reference to the
parity of $n$.
\end{remark}

\subsection{Coefficient Relations and the Vanishing of the Leading Terms}\label{subsec:D1-relations}

The rank-update formula \eqref{eq:detlemma} yields the four-term expression \eqref{eq:WA_result} for $\det(A_\omega)$,
which we label as follows:
\begin{subequations}\label{eq:D1234}
\begin{align}
\mathcal{D}_1 &= u_n
               = \sum_{i=0}^{\lfloor n/2\rfloor}\sum_{j=0}^{n-2i}\sum_{k=0}^{n-2i-j}(-1)^{i}
                 \binom{n-i}{i,j,k,n-2i-j-k}(2\lambda)^{j}a^{n-2i-j}\omega^{n-2i-j-2k}, \label{eq:D1}\\
\mathcal{D}_2 &= (\alpha^{\top} H_{\mu} e_n)\omega
              = \sum_{s=0}^{n-1}a_{n-s}u_s\,\omega^{n-s},\label{eq:D2}\\
\mathcal{D}_3 &= (e_n^{\top} H_{\mu}\alpha)\overline{\omega}
              = \sum_{s=0}^{n-1}a_{n-s}\,u_s\,\overline{\omega}^{\,n-s},\label{eq:D3}\\
\mathcal{D}_4 &= \frac{1}{u_n}\bigl[(\alpha^{\top} H_{\mu} e_n)(e_n^{\top} H_{\mu}\alpha) - (\alpha^{\top} H_{\mu}\alpha)\,u_{n-1}\bigr] \notag \\
              &= -\sum_{s=2}^{n}\sum_{t=1}^{s-1}\alpha_{s-t}\,\alpha_{n-t}\,u_{s-t-1}\,u_{t-1}\;\omega^{n-s}
                 -\sum_{s=2}^{n-1}\sum_{t=1}^{s-1} \alpha_{s-t}\,\alpha_{n-t}\,u_{s-t-1}\,u_{t-1}\;\overline{\omega}^{\,n-s}.\label{eq:D4}
\end{align}
\end{subequations}

\begin{proposition}\label{prop:Di-relation}
Assume $a_r=(-1)^r\binom{n}{r}a^r$ for $1\le r\le n$, and let
$\mathcal{D}_1,\mathcal{D}_2,\mathcal{D}_3$ be as in \eqref{eq:D1234}.
For every valid triple $(i,j,k)$ in \eqref{eq:un-expansion}, denote by
$\mathcal{D}_\ell(i,j,k)$ the contribution of that triple to the coefficient
of $\omega^{n-2i-j-2k}$ in $\mathcal{D}_\ell$, for $\ell=1,2,3$.
Then the following coefficient relations hold.

\begin{enumerate}[label=\textnormal{(\roman*)},leftmargin=*]

\item For nonnegative integers $j,k$ with $j+2k\le n-1$,
\begin{equation}\label{eq:D2-0jk}
\mathcal{D}_2(0,j,k) = -\mathcal{D}_1(0,j,k).
\end{equation}
Moreover, for integers $j$ with $0\le j\le n-2$ and $n+j$ even,
\begin{equation}\label{eq:D2-1j0}
\mathcal{D}_2(1,j,0) = \left(-1+\frac{1}{\dbinom{n-1}{j+1}}\right)\mathcal{D}_1(1,j,0).
\end{equation}

\item For nonnegative integers $j,k$ with $k\ge 1$ and $j+2k\le n$,
\begin{equation}\label{eq:D3-0jk}
\mathcal{D}_3(0,j,k) = -\mathcal{D}_1(0,j,k).
\end{equation}

\end{enumerate}
\end{proposition}

\begin{proof}
\medskip\noindent\textbf{Part (i): \eqref{eq:D2-0jk} and \eqref{eq:D2-1j0}.}

In $\mathcal{D}_2=\sum_{s=0}^{n-1}a_{n-s}u_s\,\omega^{n-s}$, the triple $(0,j,k)$ of $u_s$
contributes $\omega^{s-j-2k}$, so the net $\omega$-exponent is $(s-j-2k)+(n-s)=n-j-2k$,
independent of $s$.
By \eqref{eq:un-expansion}, the $(0,j,k)$-contribution from $u_s$ is
$\binom{s}{j,k,s-j-k}(2\lambda)^j a^{s-j-2k}$.
Factoring the multinomial coefficient as $\binom{s}{j,k,s-j-k} = \binom{j+k}{k}\binom{s}{j+k}$ and noting that $a_{n-s} = (-1)^{n-s}\binom{n}{s}a^{n-s}$, we can write $\mathcal{D}_2(0,j,k)$ by summing over $s$:
\[
\mathcal{D}_2(0,j,k) = (2\lambda)^j a^{n-j-2k} \binom{j+k}{k} \sum_{s=0}^{n-1}(-1)^{n-s}\binom{n}{s}\binom{s}{j+k},
\]
where the sum can safely start from $s=0$ since $\binom{s}{j+k}=0$ for $s < j+k$.
Because $j+k \le j+2k \le n-1 < n$, Lemma~\ref{lem:altsum-shifted} implies that the full alternating sum from $s=0$ to $n$ vanishes identically.
Therefore, the sum up to $n-1$ is the negative of the single $s=n$ term:
\[
\sum_{s=0}^{n-1}(-1)^{n-s}\binom{n}{s}\binom{s}{j+k} = -\binom{n}{n}\binom{n}{j+k} = -\binom{n}{j+k}.
\]
Substituting this back, we obtain $-\mathcal{D}_1(0,j,k)$, proving \eqref{eq:D2-0jk}.

For \eqref{eq:D2-1j0}, the triple $(1,j,0)$ of $u_s$ contributes $\omega^{s-2-j}$.
The factor $\omega^{n-s}$ in $\mathcal{D}_2$ again cancels the $s$-dependence, yielding a net exponent of $n-j-2$.
By \eqref{eq:un-expansion}, the $(1,j,0)$-contribution from $u_s$ is $-(s-1)\binom{s-2}{j}(2\lambda)^j a^{s-j-2}$.
Using the identity $(s-1)\binom{s-2}{j} = (j+1)\binom{s-1}{j+1}$, and extending the sum down to $s=1$ (since the terms vanish for $s \le j+1$), we have:
\begin{equation}\label{eq:D2-1j0-sum}
\mathcal{D}_2(1,j,0) = -(j+1)(2\lambda)^j a^{n-2-j} \sum_{s=1}^{n-1}(-1)^{n-s}\binom{n}{s}\binom{s-1}{j+1}.
\end{equation}
To evaluate this sum, we apply the Pascal iteration $\binom{s-1}{j+1} = \sum_{r=0}^{j+1}(-1)^{j+1-r}\binom{s}{r}$.
Because $j \le n-2$, the maximum index in the iteration is $j+1 \le n-1 < n$.
Thus, by Lemma~\ref{lem:altsum-shifted}, the full sum from $s=0$ to $n$ of $(-1)^{n-s}\binom{n}{s}\binom{s-1}{j+1}$ must identically vanish.
Consequently, the sum from $s=1$ to $n-1$ is just the negative of the boundary terms at $s=0$ and $s=n$:
\[
\sum_{s=1}^{n-1}(-1)^{n-s}\binom{n}{s}\binom{s-1}{j+1} = -\left( (-1)^n\binom{n}{0}\binom{-1}{j+1} + (-1)^0\binom{n}{n}\binom{n-1}{j+1} \right).
\]
Since $\binom{-1}{j+1} = \frac{(-1)(-2) \cdots(-(j+1))}{(j+1)!} = (-1)^{j+1}$, the $s=0$ term becomes $(-1)^{n+j+1}$.
Thus, the sum equals $-(-1)^{n+j+1} - \binom{n-1}{j+1} = (-1)^{n+j} - \binom{n-1}{j+1}$.
Inserting this into \eqref{eq:D2-1j0-sum}:
\[
\mathcal{D}_2(1,j,0) = -(j+1)\bigl[(-1)^{n+j}-\binom{n-1}{j+1}\bigr](2\lambda)^j a^{n-2-j}.
\]
Finally, comparing this with $\mathcal{D}_1(1,j,0) = -(j+1)\binom{n-1}{j+1}(2\lambda)^j a^{n-2-j}$ (obtained via $(n-1)\binom{n-2}{j}=(j+1)\binom{n-1}{j+1}$), and noting that $(-1)^{n+j}=1$ because $n+j$ is even, we immediately obtain \eqref{eq:D2-1j0}.

\medskip\noindent\textbf{Part (ii): \eqref{eq:D3-0jk}.}

Specializing \eqref{eq:un-expansion} to $i=0$ gives the closed form
\begin{equation}\label{eq:D1-0jk}
\mathcal{D}_1(0,j,k) = \binom{n}{j,k,n-j-k}(2\lambda)^j a^{n-j-2k}.
\end{equation}

In $\mathcal{D}_3=\sum_{s=0}^{n-1}a_{n-s}u_s\,\overline{\omega}^{\,n-s}$,
the triple $(0,j,k')$ of $u_s$ contributes $\omega^{s-j-2k'}$;
since $\overline{\omega}^{n-s}=\omega^{s-n}$, the net $\omega$-exponent is
$(s-j-2k')+(s-n)=2s-n-j-2k'$.
For this to equal the target $n-j-2k$, we need $2s-n-j-2k'=n-j-2k$, which simplifies to $s=n-(k-k')$.

Let $r = k-k'$. The index $k'\ge 0$ implies $r \le k$, and the summation bound $s \le n-1$ implies $n-r \le n-1$, so $r \ge 1$.
Thus, the target exponent is formed precisely when $s = n-r$ for $r \in \{1, 2, \dots, k\}$, corresponding to $k' = k-r$.
For each such $r$, the weight is $a_{n-s} = a_r = (-1)^r\binom{n}{r}a^r$. By \eqref{eq:un-expansion}, the $(0,j,k-r)$-contribution of $u_{n-r}$ is
$\binom{n-r}{j,k-r,n-k-j}(2\lambda)^j a^{n-j-2k+r}$.
Hence,
\[
\mathcal{D}_3(0,j,k)
= (2\lambda)^j a^{n-j}\sum_{r=1}^{k}(-1)^{r}\binom{n}{r}\frac{(n-r)!}{j!\,(k-r)!\,(n-k-j)!}.
\]
Writing $\frac{(n-r)!}{(k-r)!}=(n-k)!\binom{n-r}{k-r}$, the sum becomes
\[
\frac{(n-k)!}{j!\,(n-k-j)!}\sum_{r=1}^{k}(-1)^r\binom{n}{r}\binom{n-r}{k-r}.
\]
By the identity $\binom{n}{r}\binom{n-r}{k-r}=\binom{n}{k}\binom{k}{r}$,
\[
\sum_{r=1}^{k}(-1)^r\binom{n}{r}\binom{n-r}{k-r}
= \binom{n}{k}\sum_{r=1}^{k}(-1)^r\binom{k}{r}
= -\binom{n}{k},
\]
where the last equality follows from $\sum_{r=0}^{k}(-1)^r\binom{k}{r}=0$ for $k\ge 1$
(Lemma~\ref{lem:altsum}, \eqref{eq:altsum-1}).
Therefore
\[
\mathcal{D}_3(0,j,k)
= -\frac{(n-k)!\binom{n}{k}}{j!\,(n-k-j)!}(2\lambda)^j a^{n-j}
= -\frac{n!}{j!\,k!\,(n-j-k)!}(2\lambda)^j a^{n-j}
= -\mathcal{D}_1(0,j,k),
\]
using $\dbinom{n}{k}\dfrac{(n-k)!}{(n-k-j)!}=\dfrac{n!}{k!\,(n-k-j)!}$
and comparing with \eqref{eq:D1-0jk}. \qedhere
\end{proof}

\begin{corollary}\label{cor:leading-vanishing}
Let $A$ be the $n\times n$ companion matrix associated with $p(z)=(z-a)^n$, where $a\in\mathbb{R}$.
Then the coefficients of $\omega^n$ and $\omega^{n-1}$ in $\det(A_\omega)$ both vanish,
independently of $\lambda$ and of any geometric hypothesis on $W(A)$.
\end{corollary}

\begin{proof}
By \eqref{eq:D3}, $\mathcal{D}_3=(e_n^{\top}H_\mu\alpha)\overline{\omega}$ has positive $\omega$-degree at most $n-2$.
By Corollary~\ref{cor:D4-degree}, $\mathcal{D}_4$ also has positive $\omega$-degree at most $n-2$.
Therefore the coefficients of $\omega^n$ and $\omega^{n-1}$ in $\det(A_\omega)$ coincide with those in
$\mathcal{D}_1+\mathcal{D}_2$.

The coefficient of $\omega^n$ can only come from the triple $(0,0,0)$ in $\mathcal{D}_1+\mathcal{D}_2$.
By Proposition~\ref{prop:Di-relation}\textnormal{(i)} with $j=k=0$,
\[
\mathcal{D}_1(0,0,0)+\mathcal{D}_2(0,0,0)
= \mathcal{D}_1(0,0,0)+\bigl(-\mathcal{D}_1(0,0,0)\bigr)=0.
\]
The coefficient of $\omega^{n-1}$ can only come from the triple $(0,1,0)$ in $\mathcal{D}_1+\mathcal{D}_2$.
By Proposition~\ref{prop:Di-relation}\textnormal{(i)} with $j=1$, $k=0$,
\[
\mathcal{D}_1(0,1,0)+\mathcal{D}_2(0,1,0)
= \mathcal{D}_1(0,1,0)+\bigl(-\mathcal{D}_1(0,1,0)\bigr)=0.
\]
Hence both leading coefficients vanish.
\end{proof}

The two leading coefficients vanish identically, so the singularity condition reduces to lower Laurent exponents,
where all four terms of \eqref{eq:D1234} may contribute.

\section{Critical Coefficient Analysis and Proof of the Main Theorem}\label{sec:proof}

For $n>3$, define
\begin{equation}\label{eq:pstar}
p^*:=2(\lfloor n/2\rfloor-1)=
\begin{cases}n-2, & n\text{ even},\\ n-3, & n\text{ odd}.\end{cases}
\end{equation}
This exponent is chosen so that the highest $\lambda$-power cancels (as explained in Remark~\ref{rem:p-star-choice} below), 
leaving a polynomial in $a$ whose coefficients have fixed sign. 
The following lemma shows that at the critical exponent $p^*$, defined in \eqref{eq:pstar}, 
the possible powers of $\lambda$ are restricted by parity.

\begin{lemma}\label{lem:parity}
Let $A_\omega=2\lambda I_n-\omega B-\overline{\omega}B^{\top}$ with $B=A-aI_n$
real and $|\omega|=1$.  In the multilinear row expansion of $\det(A_\omega)$,
each monomial contributing to the coefficient of $\omega^p$ carries a factor
$\lambda^{i}$ with
\[
i\ge 0\quad\text{and}\quad i\equiv n-p\pmod{2}.
\]
Consequently, at the critical exponent $p^*$ in \eqref{eq:pstar},
\[
p^*=n-2\ \Rightarrow\ i\in\{0,2\},\quad p^*=n-3\ \Rightarrow\ i\in\{1,3\}.
\]
\end{lemma}
\begin{proof}
By multilinearity, each term of $\det(A_\omega)$ selects, row by row, one of
the three sources $2\lambda I_n$, $-\omega B$, $-\overline{\omega}B^{\top}$.
Let $i,j,k$ denote the number of rows taken from each source respectively.
Then $i+j+k=n$, and since $\overline{\omega}=\omega^{-1}$ the resulting scalar
factor contains $\lambda^{i}\omega^{j-k}$.  Fixing $j-k=p$ and substituting
$j=p+k$ into $i+j+k=n$ yields
\[
i=n-p-2k.
\]
Hence $i\equiv n-p\pmod 2$; together with $i,k\ge 0$, this leaves
$i\in\{0,2\}$ when $n-p=2$ and $i\in\{1,3\}$ when $n-p=3$. \qedhere
\end{proof}

\begin{theorem}\label{thm:main}
Let $A$ be an $n\times n$ ($n>3$) companion matrix whose spectrum consists of
a single eigenvalue $a\in\mathbb{C}$.  If $W(A)$ is a circular disk, then
$a=0$ and hence $A=J_n$.
\end{theorem}

\begin{proof}
By the rotation reduction established in Section~\ref{sec:intro}, the center of $W(A)$ is $a$ and we may assume $a\ge 0$. 
Because $A$ is a companion matrix of size $n>3$, it is not a scalar multiple of the identity, hence $W(A)$ cannot reduce to a single point; 
the disk's radius $\lambda$ is therefore strictly positive. 

Since $\det(A_\omega)=0$ for every $|\omega|=1$, each Laurent coefficient of $\det(A_\omega)$ vanishes. 
We will use this property to show that $a=0$ by contradiction. To this end, we first suppose that $a>0$.

The coefficient of $\det(A_\omega)$ at the critical exponent $p^*$ is recorded in
Table~\ref{tab:even-front} (even $n$) and Table~\ref{tab:odd-front} (odd $n$), 
where the $\mathcal{D}_4$ entries are taken directly from Lemma~\ref{lem:D4-coef}, 
using the arrays $T(n,k)$ and $S(n,k)$ defined in \eqref{eq:cd-def}.

Tables~\ref{tab:even-front} and~\ref{tab:odd-front} list each $(i,j,k)$-triple contributing to the respective coefficient,
together with the value of the contribution to the power $\omega^{n-2}$ for even $n$ and $\omega^{n-3}$ for odd $n$, respectively.
For the rows corresponding to $\mathcal{D}_1,\mathcal{D}_2,\mathcal{D}_3$, the listed triples follow the convention of
Proposition~\ref{prop:Di-relation}: they contribute to the coefficient of $\omega^{n-2i-j-2k}$ in the corresponding term.
For the $\mathcal{D}_4$ row, the displayed triple records the Chebyshev coefficient used in Lemma~\ref{lem:D4-coef}.
By Lemma~\ref{lem:parity}, only $\lambda^0,\lambda^2$ can appear at $\omega^{n-2}$ (even $n$) and
only $\lambda^1,\lambda^3$ at $\omega^{n-3}$ (odd $n$).

\begin{table}[ht]
\centering
\small
\renewcommand{\arraystretch}{1.4}
\begin{tabular}{clcc}
\toprule
Term & Contribution & $(i,j,k)$ & $\omega^{n-2}$-coefficient \\
\midrule
\multirow{3}{*}{$\mathcal{D}_1$} & \multirow{3}{*}{$u_n$}
  & $(1,0,0)$ & $-(n-1)a^{n-2}$ \\
& & $(0,2,0)$ & $2n(n-1)\lambda^{2}a^{n-2}$ \\
& & $(0,0,1)$ & $na^{n}$ \\
\midrule
\multirow{3}{*}{$\mathcal{D}_2$}
& \multirow{3}{*}{$\displaystyle\sum_{s=1}^{n-1}
  (-1)^{n-s}\tbinom{n}{s}a^{n-s}u_s\,\omega^{n-s}$}
& $(1,0,0)$ & $(n-2)\,a^{n-2}$ \\[2pt]
& & $(0,2,0)$ & $-2n(n-1)\lambda^{2}a^{n-2}$ \\[2pt]
& & $(0,0,1)$ & $-na^{n}$ \\
\midrule
$\mathcal{D}_3$ & $-na\,u_{n-1}\overline{\omega}$
  & $(0,0,1)$ & $-na^{n}$ \\
\midrule
$\mathcal{D}_4$ & (Lemma~\ref{lem:D4-coef})
  & $(0,0,0)$ & $-a^{n-2}{\displaystyle\sum_{k=2}^{n}T(n,k)\,a^{2k}}$ \\
\bottomrule
\end{tabular}
\caption{Contributions to the $\omega^{n-2}$-coefficient of $\det(A_\omega)$ for even $n$
($\mathcal{D}_1,\mathcal{D}_2,\mathcal{D}_3$: Proposition~\ref{prop:Di-relation}; $\mathcal{D}_4$: Lemma~\ref{lem:D4-coef}).}
\label{tab:even-front}
\end{table}

\begin{table}[ht]
\centering
\small
\renewcommand{\arraystretch}{1.4}
\begin{tabular}{clcc}
\toprule
Term & Contribution & $(i,j,k)$ & $\omega^{n-3}$-coefficient \\
\midrule
\multirow{3}{*}{$\mathcal{D}_1$} & \multirow{3}{*}{$u_n$}
  & $(1,1,0)$ & $-2(n-1)(n-2)\lambda a^{n-3}$ \\
& & $(0,3,0)$ & $\tfrac{4}{3}n(n-1)(n-2)\lambda^{3}a^{n-3}$ \\
& & $(0,1,1)$ & $2n(n-1)\lambda a^{n-1}$ \\
\midrule
\multirow{3}{*}{$\mathcal{D}_2$}
& \multirow{3}{*}{$\displaystyle\sum_{s=1}^{n-1}
  (-1)^{n-s}\tbinom{n}{s}a^{n-s}u_s\,\omega^{n-s}$}
  & $(1,1,0)$ & $2\bigl((n-1)(n-2)-2\bigr)\,\lambda a^{n-3}$ \\[2pt]
& & $(0,3,0)$ & $-\tfrac{4}{3}n(n-1)(n-2)\lambda^{3}a^{n-3}$ \\[2pt]
& & $(0,1,1)$ & $-2n(n-1)\lambda a^{n-1}$ \\
\midrule
$\mathcal{D}_3$ & $-na\,u_{n-1}\overline{\omega}$
  & $(0,1,1)$ & $-2n(n-1)\lambda a^{n-1}$ \\
\midrule
$\mathcal{D}_4$ & (Lemma~\ref{lem:D4-coef})
  & $(0,1,0)$ & $-2\lambda a^{n-3}{\displaystyle\sum_{k=2}^{n}S(n,k)\,a^{2k}}$ \\
\bottomrule
\end{tabular}
\caption{Contributions to the $\omega^{n-3}$-coefficient of $\det(A_\omega)$ for odd $n$
($\mathcal{D}_1,\mathcal{D}_2,\mathcal{D}_3$: Proposition~\ref{prop:Di-relation}; $\mathcal{D}_4$: Lemma~\ref{lem:D4-coef}).}
\label{tab:odd-front}
\end{table}

The $\mathcal{D}_1$ entries in both tables follow directly from \eqref{eq:un-expansion} by substituting each triple $(i,j,k)$.
The values of $\mathcal{D}_2$ and $\mathcal{D}_3$ follow from Proposition~\ref{prop:Di-relation} as multiples of the
corresponding value $\mathcal{D}_1(i,j,k)$, while the $\mathcal{D}_4$ entries are given directly by Lemma~\ref{lem:D4-coef}.

\medskip\noindent\textbf{Case 1: $n$ is even.}
For the even case, we consider the power $\omega^{p^{*}} = \omega^{n-2}$. 
Since $\det(A_\omega)=0$, the coefficient of $\omega^{n-2}$, obtained by summing Table~\ref{tab:even-front}, vanishes, i.e.,
\begin{equation}\label{eq:coef-even}
-a^{n-2}\Biggl(1+na^{2}+\sum_{k=2}^{n}T(n,k)\,a^{2k}\Biggr) = 0.
\end{equation}
Each coefficient $T(n,k)=\binom{n-1}{k-1}\binom{n}{k}$ is a strictly positive integer for all $2 \le k \le n$ and $n > 3$, 
so the parenthesized expression in \eqref{eq:coef-even} is bounded below by $1$ for all $a>0$. 
Hence $a^{n-2}=0$, which implies $a=0$, a contradiction.

\medskip\noindent\textbf{Case 2: $n$ is odd.}
For the odd case, we consider the power $\omega^{p^{*}} = \omega^{n-3}$. 
The vanishing of $\det(A_\omega)$ implies that the $\omega^{n-3}$-coefficient, summed from Table~\ref{tab:odd-front}, equals zero:
\begin{equation}\label{eq:coef-odd}
-2\lambda\,a^{n-3}\Biggl(2+n(n-1)a^{2}+\sum_{k=2}^{n}S(n,k)\,a^{2k}\Biggr) = 0.
\end{equation}
Moreover, for $2\le k\le n$ and odd $n\ge 5$,
\[
S(n,k)=\left( n-1 - \frac{k-1}{n-1} \right)\binom{n-1}{k-1}\binom{n}{k} \ge \frac{(n-1)(n-2)}{n-1} \binom{n-1}{k-1}\binom{n}{k} > 0,
\]
so the parenthesized expression in \eqref{eq:coef-odd} is bounded below by $2$ for all $a>0$. 
Since $\lambda>0$, we obtain $a^{n-3}=0$, hence $a=0$, a contradiction.

In both cases, $a=0$, which is equivalent to $A=J_n$. \qedhere
\end{proof}

\begin{remark}\label{rem:p-star-choice}
By Lemma~\ref{lem:parity}, the $\omega^{n-2}$-coefficient can a priori contain
$\lambda^0$ and $\lambda^2$ terms, and the $\omega^{n-3}$-coefficient can
contain $\lambda^1$ and $\lambda^3$ terms.  The closed forms
\eqref{eq:coef-even} and \eqref{eq:coef-odd} show that in both cases the
highest $\lambda$-power cancels: only $\lambda^0$ survives at $\omega^{n-2}$
and only $\lambda^1$ at $\omega^{n-3}$.  This cancellation, arising from
Lemma~\ref{lem:altsum-shifted}, is the reason $p^*$ is the exponent at which 
the sign argument applies directly.
\end{remark}

\section{Extension to Partial Zero Spectra}\label{sec:partial-zeros}

We next consider companion matrices whose spectrum consists of $m$ copies of $a$ and $n-m$ copies of $0$:
\begin{equation*}
\sigma(A)=\{0^{\,n-m},a^m\},\qquad 1\le m\le n-1.
\end{equation*}
The corresponding characteristic polynomial is $p(z)=z^{n-m}(z-a)^m$, and its coefficients are given by the truncated binomial formula:
\begin{equation*}
a_j=
\begin{cases}
(-1)^j\binom{m}{j} a^j, & 1\le j\le m,\\[2mm]
0, & m<j\le n.
\end{cases}
\end{equation*}

We split the analysis by the spectral multiplicity $m$.
For $m=1$, if $a=0$ there is nothing to prove. Assume $a\ne0$; then $a$ is a simple eigenvalue,
and by \cite{Wu11} it cannot be the center of $W(A)$, because the center of a circular
numerical range must be an eigenvalue whose algebraic multiplicity exceeds its geometric
multiplicity. Hence the center must be $0$, so $W(A)$ is a disk centered at the origin, and
\cite{GauWu04} gives $A = J_n$, so $a=0$, a contradiction.
For $3\le m\le n-1$, the $\omega^{n-2}$-coefficient of $\det(A_\omega)$ alone suffices to derive a contradiction; 
for $m=2$, the two coefficients $\omega^{n-2}$ and $\omega^{n-3}$ are both required.

For the cases $2\le m\le n-1$, the four-term decomposition established in Section~\ref{sec:framework} again applies:
\begin{equation*}
\det(A_\omega) = \mathcal{D}_1+\mathcal{D}_2+\mathcal{D}_3+\mathcal{D}_4.
\end{equation*}
The truncation condition $a_j=0$ for $j>m$ restricts the relevant index sets.

For $3\le m\le n-1$, the $\lambda^2$-terms in the coefficient of $\omega^{n-2}$ occur only in 
$\mathcal D_1+\mathcal D_2+\mathcal D_3$ and cancel by Lemma~\ref{lem:altsum-shifted}, 
since the relevant polynomial factors have degree $2 < m$. 
The term $\mathcal D_4$ contributes no $\lambda$-dependent part at this exponent. 
For $m=2$, this cancellation fails, and the coefficient of $\omega^{n-3}$ is also needed in order to deduce $a=0$.

\begin{proposition}\label{prop:partial}
Let $A$ be the $n\times n$ companion matrix associated with
\[
p(z) = z^{n-m}(z-a)^m, \qquad 2 \le m \le n-1.
\]
\begin{enumerate}
\item[(i)] If $3\le m\le n-1$, then the coefficient of $\omega^{n-2}$ in $\det(A_\omega)$ is
\begin{equation}\label{eq:partial-coef-i}
-a^n\left(\sum_{\ell=1}^{m}T(m,\ell)\,a^{2(\ell-1)}\right).
\end{equation}
\item[(ii)] If $m=2$, then both the coefficients of $\omega^{n-2}$ and $\omega^{n-3}$ in $\det(A_\omega)$ are given respectively by
\begin{align}
& -a^{n-2}\bigl(a^{4}+2a^{2}-4\lambda^{2}\bigr), \label{eq:partial-coef-ii-P}\\[2pt]
& -2\lambda a^{n-3}\Bigl((n-2)\bigl(a^{4}+2a^{2}-4\lambda^{2}\bigr)+2\Bigr). \label{eq:partial-coef-ii-Q}
\end{align}
\end{enumerate}
\end{proposition}

\begin{proof}
We use the decomposition \eqref{eq:D1234} with $a_j=(-1)^j\binom{m}{j}a^j$ for $1\le j\le m$ and $a_j=0$ for $j>m$.

\medskip\noindent\textbf{Part (i): $3\le m\le n-1$.}
Unlike the single-eigenvalue case, neither Proposition~\ref{prop:Di-relation} nor Lemma~\ref{lem:D4-coef}
applies directly here, since both were derived under the assumption $a_j=(-1)^j\binom{n}{j}a^j$ for all
$1\le j\le n$, whereas the truncation $a_j=0$ for $j>m$ alters the combinatorial structure.
Nevertheless, the computation of the $\omega^{n-2}$-coefficient proceeds in close analogy with the single-eigenvalue case:
the contribution of $\mathcal{D}_1$ is obtained directly from the same Chebyshev expansion as before, since
$\mathcal{D}_1=u_n$ does not depend on the companion coefficients and hence retains binomial base $n$.
For $\mathcal{D}_2$ and $\mathcal{D}_3$, the truncation restricts the summation range to $s\in\{n-m,\ldots,n-1\}$,
and Lemmas~\ref{lem:altsum}--\ref{lem:altsum-shifted} are applied with binomial base $m$ in place of $n$.
The $\mathcal{D}_4$ term is then computed from the same positive-power reduction as in Lemma~\ref{lem:D4-coef},
but with the nonzero companion coefficients truncated at degree $m$.
The resulting contributions are recorded in Table~\ref{tab:partial-front}.

\begin{table}[ht]
\centering
\small
\renewcommand{\arraystretch}{1.4}
\begin{tabular}{clcc}
\toprule
Term & Contribution & $(i,j,k)$ & $\omega^{n-2}$-coefficient \\
\midrule
\multirow{3}{*}{$\mathcal{D}_1$} & \multirow{3}{*}{$u_n$}
  & $(1,0,0)$ & $-(n-1)a^{n-2}$ \\
& & $(0,2,0)$ & $2n(n-1)\lambda^{2}a^{n-2}$ \\
& & $(0,0,1)$ & $na^{n}$ \\
\midrule
\multirow{3}{*}{$\mathcal{D}_2$}
& \multirow{3}{*}{$\displaystyle\sum_{s=n-m}^{n-1}(-1)^{n-s}\tbinom{m}{n-s}a^{n-s}u_s\omega^{n-s}$}
  & $(1,0,0)$ & $(n-1)\,a^{n-2}$ \\[2pt]
& & $(0,2,0)$ & $-2n(n-1)\lambda^{2}a^{n-2}$ \\[2pt]
& & $(0,0,1)$ & $-na^{n}$ \\
\midrule
$\mathcal{D}_3$ & $-ma\,u_{n-1}\overline{\omega}$
  & $(0,0,1)$ & $-ma^{n}$ \\
\bottomrule
\end{tabular}
\caption{Contributions to the $\omega^{n-2}$-coefficient of
$\mathcal{D}_1+\mathcal{D}_2+\mathcal{D}_3$ for $3\le m\le n-1$
(triple $(i,j,k)$ as in Proposition~\ref{prop:Di-relation}; $\mathcal{D}_4$ computed separately below).}
\label{tab:partial-front}
\end{table}

Summing Table~\ref{tab:partial-front}, the contribution of $\mathcal{D}_1+\mathcal{D}_2+\mathcal{D}_3$ 
to the coefficient of $\omega^{n-2}$ is 
\begin{equation}\label{eq:partial-front-three-app}
-m\,a^{n}.
\end{equation}

For $\mathcal{D}_4$, the same Chebyshev expansion and index substitution
as in the proof of Lemma~\ref{lem:D4-coef} apply,
with the binomial base $n$ replaced by $m$.
Indeed, in the positive-power part of \eqref{eq:D4}, the coefficient of
$\omega^{n-2}$ is obtained from the leading term of $u_{s-t-1}u_{t-1}$
and is therefore
\[
-\sum_{s=2}^{n}a^{s-2}\sum_{t=1}^{s-1}\alpha_{s-t}\,\alpha_{n-t}.
\]
Since $\alpha_q=a_{n+1-q}$, set $r=n+1-(s-t)$, $\ell=t+1$.
Then $\alpha_{s-t}=a_r$, $\alpha_{n-t}=a_\ell$, and $r-\ell=n-s$.
The truncation $a_j=0$ for $j>m$ leaves precisely $2\le\ell\le r\le m$.
Thus the contribution becomes
\[
-\sum_{\ell=2}^{m}\sum_{r=\ell}^{m}(-1)^{r+\ell}\binom{m}{r}\binom{m}{\ell}
a^{n+2(\ell-1)}.
\]
Using
\[
\sum_{r=\ell}^{m}(-1)^{r+\ell}\binom{m}{r}
=\binom{m-1}{\ell-1},
\]
we obtain
\[
-a^{n}\sum_{\ell=2}^{m}\binom{m-1}{\ell-1}\binom{m}{\ell}\,a^{2(\ell-1)}.
\]
With the notation $T(m,\ell)$ from \eqref{eq:cd-def}, this is precisely
\begin{equation}\label{eq:partial-D4-app}
-a^{n}\left(\sum_{\ell=2}^{m}T(m,\ell)\,a^{2(\ell-1)}\right).
\end{equation}
Adding \eqref{eq:partial-front-three-app} and \eqref{eq:partial-D4-app}, 
and using $T(m,1)=m$, gives \eqref{eq:partial-coef-i}.

\medskip\noindent\textbf{Part (ii): $m=2$.}
When $m=2$, the companion coefficients are $a_1 = -2a$, $a_2 = a^2$ and $a_j = 0$ for $j = 3,\ldots, n$. 
Thus the four terms in the decomposition specialize to
\begin{align*}
\mathcal{D}_1 &= u_{n},\\
\mathcal{D}_2 &= a_{2}u_{n-2}\omega^{2} + a_{1}u_{n-1}\omega 
               = a^{2}u_{n-2}\omega^{2} - 2au_{n-1}\omega,\\
\mathcal{D}_3 &= a_{2}u_{n-2}\overline{\omega}^{2} + a_{1}u_{n-1}\overline{\omega} 
               = a^{2}u_{n-2}\overline{\omega}^{2} - 2au_{n-1}\overline{\omega},\\
\mathcal{D}_4 &= -a^{4}u_{n-2}.
\end{align*}
That is
\begin{equation*}
\det(A_{\omega}) = u_n + \big(a^{2}u_{n-2}\omega^2 - 2au_{n-1}\omega\big) 
                   + \big(a^2u_{n-2}\overline{\omega}^{2} - 2au_{n-1}\overline{\omega}\big) + \big(- a^{4}u_{n-2}\big).
\end{equation*}

When $m=2$, the $\lambda^2$-terms do not cancel at a single Laurent exponent. 
Each summand in the displayed formula has the form $c\cdot u_s\cdot\omega^t$.
To compute the coefficient of $\omega^{n-p}$, where $p=2$ or $p=3$, we use the expansion \eqref{eq:D1} of $u_s$.
The relevant triple $(i,j,k)$ is determined by $2i+j+2k=p+s+t-n$. 
The tables below record the corresponding contributions for $p=2$ and $p=3$.

\smallskip\noindent\emph{The coefficient of $\omega^{n-2}$.}
For each term $c\cdot u_s\cdot\omega^t$, the relevant triple satisfies $2i+j+2k=2+s+t-n$ in \eqref{eq:un-expansion}. 

\begin{table}[ht]
\centering
\small
\renewcommand{\arraystretch}{1.4}
\begin{tabular}{clcc}
\toprule
Term & Contribution & $(i,j,k)$ & $\omega^{n-2}$-coefficient \\
\midrule
\multirow{3}{*}{$\mathcal{D}_1$} & \multirow{3}{*}{$u_n$}
  & $(1,0,0)$ & $-(n-1)a^{n-2}$ \\
& & $(0,2,0)$ & $2n(n-1)\lambda^{2}a^{n-2}$ \\
& & $(0,0,1)$ & $na^n$ \\
\midrule
\multirow{6}{*}{$\mathcal{D}_2$} & \multirow{3}{*}{$a^2u_{n-2}\omega^2$}
  & $(1,0,0)$ & $-(n-3)a^{n-2}$ \\
& & $(0,2,0)$ & $2(n-2)(n-3)\lambda^{2}a^{n-2}$ \\
& & $(0,0,1)$ & $(n-2)a^n$ \\
\cline{2-4}
& \multirow{3}{*}{$-2au_{n-1}\omega$}
  & $(1,0,0)$ & $2(n-2)a^{n-2}$ \\
& & $(0,2,0)$ & $-4(n-1)(n-2)\lambda^{2}a^{n-2}$ \\
& & $(0,0,1)$ & $-2(n-1)a^n$ \\
\midrule
$\mathcal{D}_3$ & $-2au_{n-1}\overline{\omega}$
  & $(0,0,0)$ & $-2a^n$ \\
\midrule
$\mathcal{D}_4$ & $-a^4u_{n-2}$
  & $(0,0,0)$ & $-a^{n+2}$ \\
\bottomrule
\end{tabular}
\caption{Contributions to the coefficient of $\omega^{n-2}$ in $\det(A_\omega)$ for $m=2$
(triple $(i,j,k)$ refers to the expansion of $u_s$ itself, satisfying $2i+j+2k=2+s+t-n$).}
\label{tab:coef-n2}
\end{table}

Summing Table~\ref{tab:coef-n2} by monomial type: the $a^{n-2}$ terms cancel, 
the $\lambda^2 a^{n-2}$ terms sum to $4$, the $a^n$ terms sum to $-2$, 
and $\mathcal{D}_4$ contributes the separate monomial $-a^{n+2}$. 
Therefore, the coefficient of $\omega^{n-2}$ in $\det(A_\omega)$ is
\[
4\lambda^{2}a^{n-2} - 2a^{n} - a^{n+2} = -a^{n-2}(a^4+2a^2-4\lambda^2),
\]
which is precisely \eqref{eq:partial-coef-ii-P}.

For $m=2$, only the coefficients $a_1$ and $a_2$ are nonzero. 
Consequently, the $\mathcal D_4$ contribution is concentrated in a single index pattern. 
Unlike Table~\ref{tab:even-front}, the $\lambda^2$ terms no longer cancel here, 
and thus the coefficient of $\omega^{n-3}$ is also used.

\smallskip\noindent\emph{The coefficient of $\omega^{n-3}$.}
For each term $c\cdot u_s\cdot\omega^t$, the relevant triple satisfies $2i+j+2k=3+s+t-n$ in \eqref{eq:un-expansion}.
Table~\ref{tab:coef-n3} is the $\omega^{n-3}$ counterpart of Table~\ref{tab:coef-n2} for $m=2$. 
The two resulting coefficient equations are combined in the proof of Theorem~\ref{thm:partial}.

\begin{table}[ht]
\centering
\small
\renewcommand{\arraystretch}{1.4}
\begin{tabular}{clcc}
\toprule
Term & Contribution & $(i,j,k)$ & $\omega^{n-3}$-coefficient \\
\midrule
\multirow{3}{*}{$\mathcal{D}_1$} & \multirow{3}{*}{$u_n$}
  & $(1,1,0)$ & $-2(n-1)(n-2)\lambda a^{n-3}$ \\
& & $(0,3,0)$ & $\tfrac{4}{3}n(n-1)(n-2)\lambda^3 a^{n-3}$ \\
& & $(0,1,1)$ & $2n(n-1)\lambda a^{n-1}$ \\
\midrule
\multirow{6}{*}{$\mathcal{D}_2$} & \multirow{3}{*}{$a^2u_{n-2}\omega^2$}
  & $(1,1,0)$ & $-2(n-3)(n-4)\lambda a^{n-3}$ \\
& & $(0,3,0)$ & $\tfrac{4}{3}(n-2)(n-3)(n-4)\lambda^3 a^{n-3}$ \\
& & $(0,1,1)$ & $2(n-2)(n-3)\lambda a^{n-1}$ \\
\cline{2-4}
& \multirow{3}{*}{$-2au_{n-1}\omega$}
  & $(1,1,0)$ & $4(n-2)(n-3)\lambda a^{n-3}$ \\
& & $(0,3,0)$ & $-\tfrac{8}{3}(n-1)(n-2)(n-3)\lambda^3 a^{n-3}$ \\
& & $(0,1,1)$ & $-4(n-1)(n-2)\lambda a^{n-1}$ \\
\midrule
$\mathcal{D}_3$ & $-2au_{n-1}\overline{\omega}$
  & $(0,1,0)$ & $-4(n-1)\lambda a^{n-1}$ \\
\midrule
$\mathcal{D}_4$ & $-a^4u_{n-2}$
  & $(0,1,0)$ & $-2(n-2)\lambda a^{n+1}$ \\
\bottomrule
\end{tabular}
\caption{Contributions to the coefficient of $\omega^{n-3}$ in $\det(A_\omega)$ for $m=2$
(triple $(i,j,k)$ refers to the expansion of $u_s$ itself, satisfying $2i+j+2k=3+s+t-n$).}
\label{tab:coef-n3}
\end{table}

Summing Table~\ref{tab:coef-n3} by monomial type: the $\lambda a^{n-3}$ terms 
sum to $-4$, the $\lambda^3 a^{n-3}$ terms sum to $8(n-2)$, the $\lambda a^{n-1}$ 
terms sum to $-4(n-2)$, and $\mathcal{D}_4$ contributes the separate monomial 
$-2(n-2)\lambda a^{n+1}$. Factoring, the coefficient of $\omega^{n-3}$ in 
$\det(A_\omega)$ equals
\begin{align*}
 &\ -4\lambda a^{n-3}+8(n-2)\lambda^3 a^{n-3}-4(n-2)\lambda a^{n-1}-2(n-2)\lambda a^{n+1}\\
=&\ -2\lambda a^{n-3}\Big((n-2)(a^4+2a^2-4\lambda^2)+2\Big),
\end{align*}
which is precisely \eqref{eq:partial-coef-ii-Q}.\qedhere
\end{proof}

\begin{theorem}\label{thm:partial}
Let $A$ be an $n\times n$ companion matrix with $n>3$ and spectrum
\[
\sigma(A)=\{\underbrace{0,\ldots,0}_{\scriptstyle n-m},\ \underbrace{a,\ldots, a}_{\scriptstyle  m}\},\quad 2\le m\le n-1.
\]
If $W(A)$ is a circular disk, then $a=0$.
\end{theorem}

\begin{proof}
By \cite{Wu11,CheungLi13}, the center $c$ of $W(A)$ is an eigenvalue of $A$ whose algebraic multiplicity 
is strictly greater than its geometric multiplicity. Since a companion matrix 
is nonderogatory, each eigenvalue has geometric multiplicity one, so $c$ has 
algebraic multiplicity at least two. The spectrum $\{0^{\,n-m},a^m\}$ with 
$2\le m\le n-1$ therefore restricts $c$ to $\{0,a\}$.

If $c=0$, then $W(A)$ is a circular disk centered at the origin. By \cite{GauWu04}, a companion matrix with such a numerical 
range must equal the Jordan block $J_n$. All its eigenvalues are therefore $0$, so $a=0$.

It remains to consider $c=a$. By the rotation reduction established in Section~\ref{sec:intro}, 
we may assume $a\ge 0$. Suppose for contradiction that $a>0$. Since $A$ is a companion matrix of size $n>3$, 
it is not a scalar multiple of the identity, hence $W(A)$ cannot reduce to a single point; 
the disk's radius $\lambda$ is therefore strictly positive. 
Since $\det(A_\omega)=0$ for every $|\omega|=1$, each Laurent coefficient of $\det(A_\omega)$ vanishes.

\medskip\noindent\textbf{Case~(i): $3\le m\le n-1$.} 
By Proposition~\ref{prop:partial}(i), the $\omega^{n-2}$-coefficient of $\det(A_\omega)$ is \eqref{eq:partial-coef-i}, 
and since $\det(A_\omega)=0$, we know that
\[
-a^n\left(\sum_{\ell=1}^{m}T(m,\ell)\,a^{2(\ell-1)}\right) = 0.
\]
Since $a>0$ and each $T(m,\ell)$, defined in \eqref{eq:cd-def}, is a strictly positive integer, 
the sum is bounded below by $T(m,1) = m\ge 3>0$. Thus the left-hand side cannot vanish, a contradiction.

\medskip\noindent\textbf{Case~(ii): $m=2$.} By Proposition~\ref{prop:partial}(ii), 
\eqref{eq:partial-coef-ii-P} and \eqref{eq:partial-coef-ii-Q} both vanish:
\[
-a^{n-2}(a^{4}+2a^{2}-4\lambda^{2})=0 \quad\text{and}\quad 
-2a^{n-3}\lambda\bigl((n-2)(a^{4}+2a^{2}-4\lambda^{2})+2\bigr)=0.
\]
The first equation, combined with $a>0$, gives $a^{4}+2a^{2}-4\lambda^{2}=0$. Substituting this into the 
second equation gives
\[
-4 a^{n-3}\lambda=0,
\]
which contradicts $a>0$ and $\lambda>0$.

Thus $a=0$ in all cases.\qedhere
\end{proof}

\begin{remark}[The Case $m=1$]\label{rem:m1}
When $m=1$, the spectrum is $\sigma(A)=\{0^{\,n-1},a\}$.
If $a=0$, there is nothing to prove. If $a\ne0$, then $a$ is a simple eigenvalue.
As noted at the start of this section, $a$ cannot be the center of $W(A)$ in
this case, by \cite{Wu11}, since the center of a circular numerical range must
be an eigenvalue whose algebraic multiplicity exceeds its geometric
multiplicity (here equal to $1$ for $a$, since $A$ is nonderogatory).
Hence the center must be $0$, and \cite{GauWu04} again forces $A=J_n$, so
$a=0$. Together with Theorem~\ref{thm:partial}, this shows that $a=0$ for
\emph{every} spectrum of the form $\{0^{\,n-m},a^m\}$ with $1\le m\le n-1$.
\end{remark}

\begin{remark}\label{rem:partial-vs-single}
Theorem~\ref{thm:partial} differs from Theorem~\ref{thm:main} in two structural ways.

First, for $3\le m\le n-1$, the truncation $a_j=0$ for $j>m$ replaces the binomial base $n$ 
in Lemma~\ref{lem:altsum-shifted} by $m$. 
As a result, the $\omega^{n-2}$-coefficient in the partial case takes a closed form whose
binomial factor depends only on $m$, apart from the global factor $a^n$, with no branching
by the parity of $n$; Lemma~\ref{lem:parity} therefore plays no role here.

Second, this dependence on $m$ no longer holds at $m=2$: 
the $\lambda^2$-cancellation used in the proof of Proposition~\ref{prop:partial}(i) 
requires the hypothesis $r<m$ in Lemma~\ref{lem:altsum-shifted} to hold for every $r\in\{0,1,2\}$, 
which fails at $r=m=2$. Indeed, the alternating sum then reduces to its $k=m$ term, 
$\binom{m}{m}\binom{m}{2}=1\ne 0$. A $\lambda^2$-term then remains at $\omega^{n-2}$, 
and the $\omega^{n-3}$-coefficient is needed to eliminate $\lambda$.
\end{remark}

\appendix

\section{Auxiliary Binomial Identities}\label{app:identities}

This appendix collects the auxiliary binomial identities used throughout the paper.

\begin{lemma}\label{lem:altsum}
For every integer $n\ge 1$,
\begin{equation}\label{eq:altsum-1}
\sum_{k=0}^{n}(-1)^{k}\binom{n}{k}=0,
\end{equation}
and for every integer $n\ge 2$,
\begin{equation}\label{eq:altsum-2}
\sum_{k=0}^{n}(-1)^{k}k\binom{n}{k}=0.
\end{equation}
Multiplying \eqref{eq:altsum-1} and \eqref{eq:altsum-2} through by $(-1)^n$ and isolating the $k=n$ term 
(which equals $1$ and $n$ respectively after the multiplication), we get
\begin{equation}\label{eq:altsum-1-2}
    \sum_{k=0}^{n-1}(-1)^{n-k}\binom{n}{k}=-1\quad\text{and}\quad
    \sum_{k=0}^{n-1}(-1)^{n-k}k\binom{n}{k}=-n.
\end{equation}
\end{lemma}
\begin{proof}
Identity \eqref{eq:altsum-1} is the binomial expansion of $(1-1)^{n}=0$. 
For \eqref{eq:altsum-2}, since $k\binom{n}{k}=n\binom{n-1}{k-1}$,
\[
\sum_{k=0}^{n}(-1)^{k}k\binom{n}{k}
=n\left(\sum_{k=1}^{n}(-1)^{k}\binom{n-1}{k-1}\right)
=-n\left(\sum_{j=0}^{n-1}(-1)^{j}\binom{n-1}{j}\right)=0
\]
by \eqref{eq:altsum-1} applied with $n$ replaced by $n-1\ge 1$. \qedhere
\end{proof}

\begin{lemma}\label{lem:altsum-shifted}
For integers $n\ge 1$ and $0\le r<n$,
\begin{equation}\label{eq:altsum-shifted}
\sum_{k=0}^{n}(-1)^{n-k}\binom{n}{k}\binom{k}{r}=0.
\end{equation}
\end{lemma}
\begin{proof}
Since $\binom{n}{k}\binom{k}{r}=\binom{n}{r}\binom{n-r}{k-r}$, the sum equals
\[
\binom{n}{r}\sum_{k=r}^{n}(-1)^{n-k}\binom{n-r}{k-r}
=\binom{n}{r}(-1)^{n-r}\left(\sum_{j=0}^{n-r}(-1)^{j}\binom{n-r}{j}\right) = 0,
\]
where the last equality is \eqref{eq:altsum-1} with $n$ replaced by $n-r\ge 1$.\qedhere
\end{proof}

\begin{lemma}\label{lem:D4-odd-aux}
For integers $n\ge 2$ and $0\le k\le n$,
\begin{align}
\sum_{\ell=0}^{k}(-1)^\ell\binom{n}{\ell}
&= (-1)^k\binom{n-1}{k}, \label{eq:partial-altsum}\\
\sum_{\ell=0}^{k}\ell\,(-1)^{\ell}\binom{n}{\ell}
&= (-1)^{k}\,n\binom{n-2}{k-1}, \label{eq:partial-altsum-firstmoment}
\end{align}
with the convention $\binom{n-2}{-1}:=0$ so that
\eqref{eq:partial-altsum-firstmoment} also holds at $k=0$.
Moreover, for $2\le k\le n$,
\begin{equation}\label{eq:D4-binomial-reduction}
n\binom{n-2}{k-1}+(k-2)\binom{n-1}{k-1} = (n-1)\binom{n-1}{k-1}-\binom{n-2}{k-2}.
\end{equation}
\end{lemma}

\begin{proof}
By Pascal's identity $\binom{n}{\ell}=\binom{n-1}{\ell}+\binom{n-1}{\ell-1}$,
shifting the index $i=\ell-1$ in the second sum,
\begin{align*}
\sum_{\ell=0}^{k}(-1)^{\ell}\binom{n}{\ell} 
&= \sum_{\ell=0}^{k}(-1)^{\ell}\binom{n-1}{\ell}
  +\sum_{\ell=0}^{k}(-1)^{\ell}\binom{n-1}{\ell-1}\\
&= \sum_{\ell=0}^{k}(-1)^{\ell}\binom{n-1}{\ell}
  -\sum_{i=0}^{k-1}(-1)^{i}\binom{n-1}{i}
 = (-1)^k\binom{n-1}{k},
\end{align*}
where the last equality follows since the terms $\ell=0,\ldots,k-1$ cancel, 
leaving only the $\ell=k$ term.

For \eqref{eq:partial-altsum-firstmoment}, the case $k=0$ holds by the 
convention $\binom{n-2}{-1}=0$ (both sides vanish). 
For $1\le k\le n$, since $\ell\binom{n}{\ell}=n\binom{n-1}{\ell-1}$, the index shift $j=\ell-1$ gives
\begin{align*}
\sum_{\ell=0}^{k}\,(-1)^\ell\ell\binom{n}{\ell} &= n\sum_{\ell=1}^{k}(-1)^\ell\binom{n-1}{\ell-1}
= -n\sum_{j=0}^{k-1}(-1)^{j}\binom{n-1}{j} = (-1)^{k}n\binom{n-2}{k-1}.
\end{align*}

It remains to prove \eqref{eq:D4-binomial-reduction}. The identity
\begin{equation}\label{eq:absorb-aux}
(k-1)\binom{n-2}{k-1}=(n-k)\binom{n-2}{k-2}
\end{equation}
holds for $2\le k\le n$: for $2\le k\le n-1$, both sides equal $\tfrac{(n-2)!}{(k-2)!\,(n-k-1)!}$; 
for $k=n$, both sides vanish since $\binom{n-2}{n-1}=0$ and $n-k=0$.
Expanding the left-hand side of \eqref{eq:D4-binomial-reduction} via 
Pascal's identity $\binom{n-1}{k-1}=\binom{n-2}{k-1}+\binom{n-2}{k-2}$,
\begin{align*}
n\binom{n-2}{k-1}+(k-2)\binom{n-1}{k-1}
&= (n+k-2)\binom{n-2}{k-1}+(k-2)\binom{n-2}{k-2}\\
&= (n-1)\binom{n-2}{k-1}+(k-1)\binom{n-2}{k-1}+(k-2)\binom{n-2}{k-2}\\
&\stackrel{\eqref{eq:absorb-aux}}{=}
   (n-1)\binom{n-2}{k-1}+(n-2)\binom{n-2}{k-2}\\
&= (n-1)\Bigl(\binom{n-2}{k-1}+\binom{n-2}{k-2}\Bigr)-\binom{n-2}{k-2}\\
&= (n-1)\binom{n-1}{k-1}-\binom{n-2}{k-2},
\end{align*}
where the last equality is Pascal's identity 
$\binom{n-1}{k-1}=\binom{n-2}{k-1}+\binom{n-2}{k-2}$.\qedhere
\end{proof}

\bibliographystyle{plain}
\bibliography{Circular_Ref}

@article{Calbeck08,
  author    = {Calbeck, William},
  title     = {Elliptic numerical ranges of $3\times 3$ companion matrices},
  journal   = {Linear Algebra Appl.},
  volume    = {428},
  year      = {2008},
  pages     = {2715--2722}
}

@article{ChenLouck96,
  author    = {Chen, William Y. C. and Louck, James D.},
  title     = {The combinatorial power of the companion matrix},
  journal   = {Linear Algebra Appl.},
  volume    = {232},
  year      = {1996},
  pages     = {261--278}
}

@article{CheungLi13,
  author    = {Cheung, Wai-Shun and Li, Chi-Kwong},
  title     = {Elementary proofs for some results on the circular symmetry of the numerical range},
  journal   = {Linear Multilinear Algebra},
  volume    = {61},
  year      = {2013},
  pages     = {643--647}
}

@article{daFonPetro01,
  author    = {da Fonseca, Carlos M. and Petronilho, Jos{\'e} C.},
  title     = {Explicit inverses of some tridiagonal matrices},
  journal   = {Linear Algebra Appl.},
  volume    = {325},
  year      = {2001},
  pages     = {7--21}
}

@article{GauWu04,
  author    = {Gau, Hwa-Long and Wu, Pei Yuan},
  title     = {Companion matrices: reducibility, numerical ranges and similarity to contractions},
  journal   = {Linear Algebra Appl.},
  volume    = {383},
  year      = {2004},
  pages     = {127--142}
}

@book{GauWu21,
  author    = {Gau, Hwa-Long and Wu, Pei Yuan},
  title     = {Numerical Ranges of Hilbert Space Operators},
  publisher = {Cambridge University Press},
  address   = {New York},
  year      = {2021}
}

@book{Harville97,
  author    = {Harville, David A.},
  title     = {Matrix Algebra From a Statistician's Perspective},
  publisher = {Springer-Verlag},
  address   = {New York},
  year      = {1997}
}

@book{HornJohnson91,
  author    = {Horn, Roger A. and Johnson, Charles R.},
  title     = {Topics in Matrix Analysis},
  publisher = {Cambridge University Press},
  address   = {Cambridge},
  year      = {1991}
}

@article{Kippenhahn51,
  author = {Kippenhahn, R.},
  title  = {{\"U}ber den {Wertevorrat} einer {Matrix}},
  journal= {Math. Nachr.},
  volume = {6},
  year   = {1951},
  pages  = {193--228},
  note   = {English translation by P. F. Zachlin and M. E. Hochstenbach: 
            \emph{On the numerical range of a matrix}}
}

@book{MasonHandscomb2002,
  author    = {Mason, John C. and Handscomb, David C.},
  title     = {Chebyshev Polynomials},
  publisher = {CRC Press},
  address   = {Boca Raton, FL},
  year      = {2002}
}

@misc{OEIS103371,
author       = {{OEIS Foundation Inc.}},
title        = {{The On-Line Encyclopedia of Integer Sequences}},
howpublished = {\url{https://oeis.org/A103371}},
note         = {Sequence {A103371}},
}

@article{Wu11,
  author    = {Wu, Pei Yuan},
  title     = {Numerical ranges as circular discs},
  journal   = {Appl. Math. Lett.},
  volume    = {24},
  year      = {2011},
  pages     = {2115--2117}
}

\end{document}